\documentclass{article}
\usepackage{amsthm}
\usepackage{amsmath}
\usepackage{amssymb}
\usepackage{graphicx}
\newtheorem{theorem}{Theorem}[section]
\newtheorem*{theorem*}{Theorem}
\newtheorem*{lemma*}{Lemma}
\newtheorem{corollary}{Corollary}[theorem]
\newtheorem{lemma}[theorem]{Lemma}
\theoremstyle{definition}

\newtheorem{prop}[theorem]{Proposition}
\usepackage{arxiv}
\usepackage{xcolor}
\usepackage[utf8]{inputenc} 
\usepackage[T1]{fontenc}    
\usepackage{hyperref}       
\usepackage{url}            
\usepackage{booktabs}       
\usepackage{amsfonts}       
\usepackage{nicefrac}       
\usepackage{microtype}      
\usepackage{enumitem}
\usepackage{comment}
\usepackage{wrapfig}
\usepackage{mwe}
\usepackage{accents}
\newcommand{\ubar}[1]{\underaccent{\bar}{#1}}
\numberwithin{equation}{section}
\DeclareMathOperator{\Tr}{Tr}
\newcommand{\R}{\mathbb{R}}
\newcommand{\C}{\mathbb{C}}
\newcommand{\M}{\mathbb{M}}
\newcommand{\mba}{\mathbf{a}}
\newcommand{\mbb}{\mathbf{b}}
\newcommand{\Diag}{\text{Diag}}
\newcommand{\mbx}{\mathbf{x}}
\newcommand{\mby}{\mathbf{y}}

\newcommand{\sgn}{\text{sgn}}

\newcommand{\s}{\textcolor{white}{........}}
\title{Reverse H\"{o}lder, Minkowski, And Hanner Inequalities For Matrices}

\author{
\textbf{Victoria M~Chayes}\\
Department of Mathematics\\
Rutgers University\\
Piscataway, NJ 08854 \\
\texttt{vc362@math.rutgers.edu}
}

\begin{document}

\maketitle

\begin{abstract}
We examine a number of known inequalities for $L^p$ functions with reverse representations for $s<1$ with complex matrices under the $p$-norms $||X||_p=\Tr[(X^\ast X)^{p/2}]^{1/p}$, and similarly defined quasinorm or antinorm quantities  $||X||_s=\Tr[(X^\ast X)^{s/2}]^{1/s}$. Analogous to the reverse H\"{o}lder and reverse Minkowski for $L^p$ functions, it has recently been shown that for $A,B\in M_{n\times n}(\C)$ such that $|B|$ is invertible, $||AB||_1\geq ||A||_{s}||B||_{s/(s-1)}$ and for $A,B$ positive semidefinite that $||A+B||_s\geq ||A||_s+||B||_s$. We comment on variational representations of these inequalities. A third very important inequality is Hanner's inequality $||f+g||_p^p+||f-g||_p^p\geq(||f||_p+||g||_p)^p+|||f||_p-||g||_p|^p$ in the $1\leq p\leq 2$ range, with the inequality reversing for $p\geq 2$. The analogue inequality  has been proven to hold matrices in certain special cases. No reverse Hanner has established for functions or matrices considering ranges with $s<1$. We develop a reverse Hanner inequality for functions, and show that it holds for matrices under special conditions; it is sufficient but not necessary for $C+D, C-D\geq 0$. We also extend certain related singular value rearrangement inequalities that were previously known in the $1\leq p\leq3$ range to the $s<1$ range. Finally, we use the same techniques to characterize the previously unstudied equality case: we show that there is equality when $p\neq 1,2$ if and only if $|D|=k|C|$, which is directly analogous to the $L^p$ equality condition.

\keywords{Hanner's Inequality \and Reverse Holder Inequality \and Reverse Minkowski Inequality  \and p-Schatten Norm \and Uniform Convexity \and Matrix Inequality \and Majorization}
\end{abstract}

\section{Introduction}
\label{intro}
\s

It is of great interest to generalize equalities known for $L^p$ functions to complex matrices. The H\"{o}lder and Minkowski inequalities are well-known inequalities that are fundamental to the study of $L^p$ spaces. The $p$-Schatten norm  $||X||_p=\Tr[(X^\ast X)^{p/2}]^{1/p}$ is known to also satisfy these inequalities when $p\geq 1$.  Using the technique of majorization, \cite{BOURIN201422} first established a reverse Minkowski inequality, \begin{equation}\label{reversem}
||A+B||_s\geq ||A||_s+||B||_s
\end{equation}
for $A,B>0$ and $s<1$, and in \cite{BOURIN201422} \cite{SHI2021257} a reverse H\"{o}lder Inequality \begin{equation}\label{reverseh}
||AB||_1\geq ||A||_s||B||_r
\end{equation}
with $B$ invertible, $0<s<1$, and $r=\frac{s}{s-1}$. This comes from the more general inequality \begin{equation}
||| |AB|^r |||^{1/r}\geq ||| |A|^p |||^{1/p}||| |B|^q |||^{1/q}
\end{equation}
for $r , p > 0$ , $q < 0$, and $\frac{1}{r}=\frac{1}{p}+\frac{1}{q}$, where $|||\cdot|||$ is any unitarily invariant norm. 

\s The p-norm notably has the dual representation \begin{equation} 
||A||_p=\sup_{B \colon ||B||_q=1} |\Tr[AB^\ast]|.
\end{equation}
where $p>1$, $\frac{1}{p}+\frac{1}{q}=1$.

We establish using Holder's inequality a dual representation of the quasinorm $(0<s'<1)$ and antinorm $(0<s')$  \begin{equation}
	||A||_s=\inf_{B\colon ||B||_r=1} ||AB^\ast||_1=\inf_{B\colon ||B||_r=1} \Tr[|AB^\ast|]
\end{equation}
where here $s<1$, and $r=\frac{s}{s-1}$. We then show how this can be used to independently establish Equation (\ref{reversem}), as well as other inequalities that can be derived from variational representations both for H\"{o}lder and reverse H\"{o}lder inequalities.

\s It has also been of great interest to extend Hanner's Inequality for $L^p$ spaces \begin{equation}
||f+g||_p^p+||f-g||_p^p\geq (||f||_p+||g||_p)^p+|||f||_p-||g||_p|^p
\end{equation}
for $1\leq p\leq 2$ to the non-communative analogue in $C^p$ \begin{equation}\label{h2}
||X+Y||_p^p+||X-Y||_p^p\geq (||X||_p+||Y||_p)^p+|||X||_p-||Y||_p|^p.
\end{equation}

\s Hanner's inequality was originally proven as a simpler manner of proving uniform convexity of $L^p$ spaces, and in fact establishes the exact modulus of smoothness and convexity. The generalization to $C^p$ convexity was first addressed in \cite{Tomczak1974}, and the optimal coefficients of 2-uniform smoothness convexity were determined in \cite{Ball1994}. However, the question of a general Hanner's Inequality holding remained open, and has only been proven in the following special cases:

\begin{enumerate}
\item In \cite{mccarthy1967} (McCarthy, 1967): for all $X,Y\in M_{n\times n}(\C)$ such that $||X||_p=||Y||_p$, and all ranges of $p$. Note the proof for $1\leq p\leq 2$ given in \cite{mccarthy1967} is incorrect and corrected in \cite{hirzallah2002non}.
\item In \cite{Tomczak1974} (Tomczak-Jaegermann 1974): for all $X,Y\in \M_{n\times n}(\C)$ and $p=2k$.
\item In \cite{Ball1994} (Ball, Carlen, Lieb 1994): for all $X,Y\in \M_{n\times n}(\C)$ and $1\leq p\leq 4/3$, and $p\geq 4$; and for $X+Y, X-Y\geq 0$ and all ranges of $p$.
\item In \cite{Chayes2020} (Chayes, 2020): for all self-adjoint $X,Y\in \M_{n\times n}(\C)$ such that the anticommutator $\{X,Y \}=XY+YX=0$, and all ranges of $p$.
\end{enumerate}

\s Hanner's inequality can be seen as a correction term for the Minkowski inequality. Therefore, after examining the reverse Minkowski inequality, it is natural to ask if a Hanner-like inequality can also be established for $s<1$. We prove the following new inequalities for this range:

\begin{theorem}\label{RHFV}
Let $\mbx,\mby\in \R^n$. Then \begin{equation}
||\mbx+\mby||_s^s+||\mbx-\mby||_s^s	\leq (||\mbx||_s+||\mby||_s)^s+\big|||\mbx||_s-||\mby||_s\big|^s,
\end{equation}
for $0<s<1$, with the inequality reversing when $s<0$.
\end{theorem}
For $0<s<1$, this reverse inequality for vectors implies a reverse Hanner for $s$-integrable functions $f,g$ by applying dominated convergence to pointwise limits of simple functions approximating $f,g$ from below. For $s<0$, a reverse Hanner for general $s$-integrable functions can be shown with the original convexity argument used by Hanner \cite{hanner1956}. 

\s For the matrix case,  we establish \begin{theorem}\label{tHR}
Let $C+D, C-D\geq 0$. Then \begin{equation}\label{HR}
||C+D||_s^s+||C-D||_s^s.\leq	(||C||_s+||D||_s)^s+(||C||_s-||D||_s)^s
\end{equation}
for $0<s<1$, with the inequality reversing when $s<0$. There is equality only if $C$ and $D$ commute. However, the relationship of Equation (\ref{HR}) may not occur for general self-adjoint $X,Y$. In particular, if one does not have $||X\pm Y||_s^s\leq ||X+\pm Y_{\Diag}||_s^s$ for $0<s<1$ and reversing for $s<0$, or neither $||X_\Diag||_s\leq ||Y||_s$ nor $||Y_\Diag||_s\leq||X||_s$ holds, then the relationship may not hold.
\end{theorem}

\s Letting $\sigma_\uparrow(X)$ and $\sigma_\downarrow(X)$ denote the singular values of $X$ in ascending and descending order respectively, we can also examine the singular value rearrangement inequalities first studied in \cite{Carlen2006} as a potential method to extend the known cases of Hanner's Inequality to matrices. While it was proven in \cite{Chayes2020} that these inequalities could not be extended to all matrices, and hence could not be used to extend Hanner's Inequality to matrices, they are still of interest to investigate. We show that under the same conditions that the inequalities are known to hold for $1\leq p\leq 3$, we can also establish inequalities for $s<1$: 

\begin{theorem}\label{svr1}
Let $C,D\in M_{n\times n}(\C)$ with $C\geq |D|\geq 0$ and $\sigma_n(C)\geq \sigma_1(D)$. Then \begin{equation}	
||C+D||_s^s+||C-D||_s^s\geq ||\sigma_\uparrow(C)+\sigma_\downarrow(D)||_s^s+||\sigma_\uparrow(C)-\sigma_\downarrow(D)||_s^s
\end{equation}
for $0<s<1$, with the inequality reversing for $s<0$. 

\end{theorem}
\begin{theorem}\label{svr2}
Let $C,D\in M_{n\times n}(\C)$ with $C\geq D\geq 0$. Then \begin{equation}	
||C+D||_s^s+||C-D||_s^s\leq ||\sigma_\uparrow(C)+\sigma_\uparrow(D)||_s^s+||\sigma_\uparrow(C)-\sigma_\uparrow(D)||_s^s
\end{equation}
for $0<s<1$, with the inequality reversing for $s<0$. 
\end{theorem}	
For each of the theorems, the conditions on $C$ and $D$ are necessary.

\s Finally, while the rest of this paper deals with inequalities for values $s<1$, the techniques developed can be used to prove new properties for expressions in the $p\geq 1$ range. In particular, we are able to use the same techniques to characterize for the first time the necessary conditions for equality in Hanner's inequality for positive semidefinite matrices:

\begin{theorem}\label{HE}
For two matrices $C,D\in M_{n\times n}(\C)$ that are self-adjoint with $C+D, C-D\geq 0$, then \begin{equation}
||C+D||_p^p+||C-D||_p^p=(||C||_p+||D||_p)^p+(||C||_p-||D||_p)^p)
\end{equation}
for $p>1, p\neq 2$ if and only if $|D|=kC$. For any two matrices $C,D\in M_{n\times n}(\C)$
and $1<p\leq\frac{4}{3}$ or $p>4$, \begin{equation}
	||C+D||_p^p+||C-D||_p^p=(||C||_p+||D||_p)^p+(||C||_p-||D||_p)^p)
\end{equation}
if and only if $|D|=k|C|$.
\end{theorem}

\s We give a background to majorization and prove foundational lemmas for the rest of the paper  in Section 2, comment on reverse H\"{o}lder and Minkowski inequalities in Section 3, examine reverse Hanner in Section 4, singular value rearrangement inequalities in Section 5, and address the equality case for Hanner's inequality in Section 6. We use the following notation: $\sigma(X)$ denotes the singular values of a matrix $X$ in descending order unless otherwise indicated; for a vector $\mbx$, $[\mbx]:=[\Diag(\mbx)]$ is the matrix with the vector along the diagonal; and for a matrix $X$, $X_\Diag :=[\Diag(X_{ii})]$ is the matrix with only non-zero entries of the diagonal elements of $X$.

\section{Majorization Background And Equality Characterizations}

\s Let $\mba,\mbb\in\R^n$ with components labeled in descending order $a_1\geq\dots\geq a_n$ and $b_1\geq\dots\geq b_n$. Then $\mathbf{b}$ weakly majorizes $\mathbf{a}$, written $\mathbf{a}\prec_{w} \bf{b}$, when \begin{equation}
\sum_{i=1}^ka_i\leq \sum_{i=1}^k b_i, \qquad 1\leq k \leq n
\end{equation}
and majorizes $\mba\prec\mbb$ when the final inequality is an equality. Weak log majorization $\mba\prec_{w(\log)}\mbb$ is similarly defined for non-negative vectors as \begin{equation}
\prod_{i=1}^ka_i\leq \prod_{i=1}^k b_i, \qquad 1\leq k \leq n
\end{equation}
with log majorization $\mba\prec_{(\log)}\mbb$ when the final inequality is an equality. Note that it is not necessary that the vectors $\mba$ and $\mbb$ be in descending order for majorization or log majorization---majorization is explicitly defined with respect the the rearrangements of the values in descending order. We assume that vectors are labelled in descending order wholly for ease of notation. 

\s Majorization can be characterized by doubly stochastic matrices as follows: there is majorization $\mba\prec\mbb$ if and only if there exists a double stochastic matrix $D$ such that $\mba=D\mbb$. There is weak majorization $\mba\prec_{w}\mbb$ for non-negative vectors $\mba$ and $\mbb$ if and only if there exists a doubly substochastic matrix $K$ such that $\mba=K\mbb$ \cite{Hardy1929} \cite{MajBook}.

\s To see the relationship to convexity, we return the most vital property of majorization. Suppose $\mba\prec_{w}\mbb$. Then for any function $\phi:\R\rightarrow\R$ that is increasing and convex on the domain containing all elements of $\mba$ and $\mbb$, \begin{equation}
\sum_{i=1}^n \phi(a_i)\leq\sum_{i=1}^n \phi(b_i).
\end{equation}
If $\mba\prec\mbb$, the `increasing' requirement can be dropped \cite{Hardy1929} \cite{Hardy2} \cite{tomic1949} \cite{Weyl1949}.

\s One can immediately see that a power identity follows: \begin{lemma}\label{weakpower}
	Let $\mba,\mbb\in\R_n^+$. Suppose $\mba\prec_{w}\mbb$. Then $\mba^s\prec_{w}\mbb^s$ for all $s\geq 1$.
\end{lemma}
The majorization identity proven by the author in \cite{Chayes2020} will also be of use to us:
\begin{lemma}\label{weaksum}
	Let $\mbx\prec_w\mby$, and $\mba\prec_{w}\mbb$ be non-negative vectors labeled in descending order.  Then $\mbx\mba\prec_w\mby\mbb$.
\end{lemma}

\s We will need the following theorem as well to characterize equality cases of majorization with strictly convex functions; an analogous theorem for log majorization was proven in \cite{GTEqualityCases}. 

\begin{theorem}\label{strictmaj}
Let $\phi:\R\rightarrow\R$ be a strictly convex function. Then $\mba\prec\mbb$ and $\sum_{i=1}^n\phi(a_i)=\sum_{i=1}^n\phi(b_i)$ implies $\mba=\Theta\mbb$ for some permutation matrix $\Theta$.
\end{theorem}
\begin{proof}
Let $\mba\prec \mbb$. Then $\mba=D\mbb$ for some doubly stochastic matrix $D$. By  Birkhoff's theorem \cite{birkhoff1946tres}, $D$ can be written as the weighted sum of permutation matrices $D=\sum_{i=1}^k p_i\Theta_i$, with $0\leq p_i\leq 1$ and $\sum_{i=1}^k p_i=1$. For a strictly convex $\phi$, then 
\begin{equation}
\sum_{i=1}^n \phi\left(\sum_{j=1}^kp_j b_{j_i} \right)<\sum_{i=1}^n \phi(b_i),
\end{equation}
if $D$ is not a pure permutation matrix. 
\end{proof}

\s Finally, there are many known majorization relationships that matrices hold that we will make use of throughout the proofs of our theorems. We will need the well-known majorization relationship between a Hermitian matrix's diagonal elements and its eigenvalues:

\begin{theorem}{Schur \cite{schur1923uber}; Mirsky \cite{mirsky1957}}\label{matdiag}
Let $X\in \M_{n\times n}(\C)$ be a self-adjoint matrix with diagonal elements $\mathbf{x}:=(x_{11},\dots, x_{nn})$. Then \begin{equation}
\mathbf{x}\prec \lambda(X)
\end{equation}
\end{theorem}

and the relationship between the singular values of products of matrices:

\begin{theorem}{(Horn \cite{horn1950singular}; Gel'fand and Naimark \cite{GelNai50})}\label{sabweak} 
Let $A,B\in M_{n\times n}(\C)$. Then \begin{equation}
\sigma_\uparrow(A)\sigma_\downarrow(B)\prec_{(\log)}	\sigma(AB)\prec_{(\log)}\sigma(A)\sigma(B).
\end{equation}
\end{theorem}

With these tools, we are ready to prove the main theorems of the paper.

\section{Comments On Reverse H\"{o}lder And Minkowski Inequalities For Matrices}

\s We first show that a dual representation for quasinorms holds:

\begin{lemma}\label{duality}
Let $A\in M_{n\times n}(\C)$ be invertible, $s<1$, and $r=\frac{s}{s-1}$. Then \begin{equation}
||A||_s=\inf_{B\colon ||B||_r=1} ||AB^\ast||_1=\inf_{B\colon ||B||_r=1} \Tr[|AB|]
\end{equation}
and there exists a matrix $B$ such that the infimum is reached.
\end{lemma}
\begin{proof}
Without loss of generality, we can address the case where $A,B>0$: as $||\cdot||_1$ is unitarily invariant, then by polar decomposition, \begin{equation}
||AB^\ast||_1=||U|A||B|V^\ast||_1=|||A||B|||_1
\end{equation}
For all $B>0$, $||B||_r=1$, by the Araki-Lieb-Thirring inequality and Reverse H\"{o}lder for matrices, we have \begin{equation}
\Tr[|AB|]=\Tr[(BA^2B)^{\frac{1}{2}}]\geq \Tr[B^{1/2}AB^{1/2}]=\Tr[AB]\geq ||A||_s||B||_r=||A||_s	
\end{equation}
What remains is finding a suitible matrix such that the infimum is reached. We claim that \begin{equation}
B=\frac{A^{s-1}}{||A||_s^{s-1}}
\end{equation}
does this.

\s We note that $B>0$, and \begin{equation}
||B||_r=\frac{\left(\sum_{i=1}^n	(\lambda_i(A)^{s-1})^{s/(s-1)} \right)^{(s-1)/s}}{\left(\sum_{i=1}^n\lambda_i(A)^s \right)^{(s-1)/s}}=	\frac{\left( \sum_{i=1}^n\lambda_i(A)^{s} \right)^{(s-1)/s}}{\left(\sum_{i=1}^n\lambda_i(A)^s \right)^{(s-1)/s}}=1.
\end{equation}	
Then as \begin{equation}
\Tr[AB]=||A||_s^{1-s}\Tr[AA^{s-1}]=\frac{\sum_{i=1}^n	\lambda(A)^s}{\left(\sum_{i=1}^n\lambda_i(A)^s \right)^{(s-1)/s}}=\left(\sum_{i=1}^n\lambda_i(A)^s \right)^{1/s}=||A||_s
\end{equation}
we see the infimum is reached as desired.
\end{proof}

This allows us to prove quickly reverse Minkowski without the need for majorization: 

\begin{proof}{Proof of Equation (\ref{reversem})}

\s Using the representation of Lemma \ref{duality} and Theorem \ref{reverseh}, \begin{equation}
||A+B||_s=\Tr[(A+B)Y]=\Tr[AY]+\Tr[BY]\geq ||A||_s||Y||_r+||B||_s||Y||_r=||A||_s+||B||_s
\end{equation}
\end{proof}

\s We now comment on variational formulations for both H\"{o}lder and reverse H\"{o}lder inequalities. The remainder of this section originated in a correspondence with J.C. Bourin; the author thanks him for his remarks, and for permission to include the following propositions and proofs.

\s For $\alpha>0$,  $t\mapsto \| |A^t ZB^t|^{\alpha}\|$ is a log-convex function on $(-\infty,\infty)$. This is equivalent (see \cite{Bourin2}[Corollary 3.2]) to the general H\"{o}lder inequality with any unitarily invariant norm $|||\cdot |||$ for the functional $X\mapsto ||||X|^{\alpha}|||$ and congugate exponents $p>1$, $q=\frac{p}{p-1}$,
\begin{equation}\label{holder}
||| |XY|^{\alpha} ||| \leq ||| |X|^{\alpha p} |||^{1/p} ||| |Y|^{\alpha q} |||^{1/q}.
\end{equation}
Of course, as for functions, we have an equality  for any $A\geq 0$ and $X'=A^{1/p}$, $Y'=A^{1/q}$: \begin{equation}
||| |X'Y'|^{\alpha} ||| = ||| |X'|^{\alpha p} |||^{1/p} ||| |Y'|^{\alpha q} |||^{1/q}= \frac{1}{p} ||| |X'|^{\alpha p}||| +  \frac{1}{q} ||| |Y'|^{\alpha q}||| .
\end{equation}
Thus we can state the following variational version of \eqref{holder}:

\begin{prop}\label{prop1} Let $A\in\M_{n\times n}(\C)$, $\alpha >0$, $p>1$, $q=\frac{p}{p-1}$. Then, for all unitarily invariant norms, \begin{equation}
||| |A|^{\alpha} ||| = \min_{A=BC}\left\{||| |B|^{\alpha p} |||^{1/p} ||| |C|^{\alpha q} |||^{1/q}\right\}=  \min_{A=BC}\left\{\frac{1}{p} ||| |B|^{\alpha p}||| +  \frac{1}{q} ||| |C|^{\alpha q}||| \right\}.
\end{equation}
\end{prop}

Another slight variation is the following (clearly with $Z=I$, we have\eqref{holder}):

\begin{prop}\label{prop2} Let $X,Y\in M_{n\times n}(\C)$, $\alpha >0$, $p>1$, $q=\frac{p}{p-1}$. Then, for all unitarily invariant norms,
\begin{equation}
||| |XY|^{\alpha} ||| = \inf_{Z \  invertible}\left\{||| |XZ|^{\alpha p} |||^{1/p} ||| |Z^{-1}Y|^{\alpha q} |||^{1/q}\right\}=  \inf_{Z \  invertible}\left\{\frac{1}{p} ||| |XZ|^{\alpha p}||| +  \frac{1}{q} ||| |Z^{-1}Y|^{\alpha q}||| \right\}.
\end{equation}
\end{prop}

If $X,Y$ are both invertible, then any decompostion $XY=BC$ is achieved with some invertible $Z$, $B=XZ$, $C=Z^{-1}Y$, hence Proposition 1.1 entails Proposition 1.2 and the infimum is a minimum.

\s Proposition \ref{prop1} still holds for infinite dimensional operators but Proposition \ref{prop2} does not hold for infinite dimensional operators: For instance if $X$ is a rank one projection and $Y$ is the identity, then 
\begin{equation}
\Tr |XY| =1 \quad {\mathrm{and}}  \quad \inf_{Z \  invertible}\left\{\frac{1}{p} ||| |XZ|^{\alpha p}||| +  \frac{1}{q} ||| |Z^{-1}Y|^{\alpha q}||| \right\} =\infty.\end{equation}

\s Next, one supposes that $0<p<1$, and let $r=\frac{-p}{p-1}$. Then from \eqref{holder} one may derive the reverse Holder inequality:
\begin{equation}\label{revholder}
||| |XY|^{\alpha} ||| \ge  ||| |X|^{\alpha p} |||^{1/p} ||| |Y|^{-\alpha r} |||^{-1/r}.
\end{equation}
where for a noninvertible $Y$ we define (by continuity) $\| |Y|^{-r} \|^{-1/r}=0$. We have an obvious equality case, for any $A>0$ and $X'=A^{1/p}$, $Y'=A^{-1/r}$,  and thus we can state:

\begin{prop}\label{prop3} Let $A\in M_{n\times n}(\C)$ be invertible, $\alpha >0$, $0<p<1$, $r=\frac{-p}{p-1}$. Then, for all unitarily invariant norms,
\begin{equation}
||||A|^{\alpha} ||| = \max_{A=BC}\left\{||| |B|^{\alpha p} |||^{1/p} ||| |C|^{-\alpha r} |||^{-1/r}\right\}=  \max_{A=BC}\left\{\frac{1}{p} ||| |B|^{\alpha p}||| - \frac{1}{r} ||| |C|^{-\alpha r}||| \right\}.
\end{equation}
\end{prop}

\s Further direct proofs of Reverse H\"{o}lder can be found in \cite{SHI2021257} and with similar methods, in an unpublished manuscript of Bourin-Hiai (the ArXiv version of \cite{BOURIN201422} ), which we will not repeat here; the key observation is the majorization of Theorem \ref{sabweak}. One can further take a refinement of Araki and the Gel'fand-Naimark log-majorizations for $r\in(0,1)$ and $A,B>0$ of 
\begin{equation}
||||A^{1/r}B^{1/r}|^r ||| \ge ||| AB||| \ge  |||A^rB^r|^{1/r} |||	\ge ||| [\sigma_\downarrow(A)][\sigma_\uparrow(B)]|||.
\end{equation}
We use this relation to derive the next corollary. We denote by $A\beta_0 B$ the geometric mean of two positive matrices, \begin{equation}
A\beta_0 B :=\exp(\{\log A +\log B\}/2).
\end{equation}

\begin{corollary}
	For every invertible $A,B \in M_{n\times n}(\C)^{+}$ and every $k=1,\dots,n$,
\begin{equation}
	{1\over k}\sum_{j=1}^k\lambda_j(A\,\beta_0\,B)
	\ge \Biggl\{\prod_{j=1}^k\lambda_j(A)\Biggr\}^{1/2k}
	\Biggl\{\prod_{j=1}^k\lambda_{n+1-j}(B)\Biggr\}^{1/2k}.
\end{equation}
\end{corollary}

\begin{proof}
As remarked above  we have 
\begin{equation}
||||A^rB^r|^{1/r} ||| \ge ||| A^p|||^{1/p}||| B^q|||^{1/q}
\end{equation}
for all unitarily invariant norms, $0<p<1$, $\frac{1}{p}+\frac{1}{q}=1$, and $r\in(0,1)$. The Lie-Trotter formula
says that $\lim_{r\searrow 0} |A^rB^r|^{1/r}=\exp(\log A + \log B)$, and thus
\begin{equation}
|||\exp(\log A + \log B)||| \ge ||| A^p|||^{1/p}||| B^q|||^{1/q}.
\end{equation}
Letting $|||\cdot|||=k^{-1}\|\cdot\|_{(k)}$
(the Ky Fan norm) and $p\searrow0$ ($q\searrow -\infty$)
we obtain
\begin{equation}
{1\over k}\sum_{j=1}^k\lambda_j(\exp(\log A + \log B))
\ge \Biggl\{\prod_{j=1}^k\lambda_j(A)\Biggr\}^{1/k}
\Biggl\{\prod_{j=1}^k\lambda_{n+1-j}(B)\Biggr\}^{1/k}
\end{equation}
for $k=1,\dots,n$. The result follows by replacing $A,B$ by $A^{1/2}, B^{1/2}$.
\end{proof}

\section{The Reverse Hanner Inequalities For Functions And Matrices}

\s There are many ways to show that a reverse Hanner inequality holds for functions and matrices. It should be noted that the original proof of Hanner's inequality using a convexity argument does apply in the $s<1$ case. We will use a new method involving the Talyor expansions of both sides of the inequality, because this is the methodology needed for the matrix case. 

\begin{proof}{Proof of Theorem \ref{RHFV}.} 

\s We will assume without loss of generality by labling choice that $||\mbx||_s> ||\mby||_s$; then the $||\mbx||_s= ||\mby||_s$ case can be deduced through continuity. We can expand the Taylor series \begin{equation}
(||\mbx||_s+r||\mby||_s)^s+(||\mbx||_s-r||\mby||_s)^s=2\sum_{k=0}^\infty \frac{(s)_{2k}}{(2k)!}||\mby||_s^{2k}||\mbx||_s^{s-2k}r^{2k}.	
\end{equation}
where $(s)_{k}=s(s-1)\dots (s+1-k)$ denotes the rising Pochhammer symbol. 

\s This Taylor series will always converge at $r=1$, the first coefficient is always positive, and afterwards when $0<s<1$ all coefficients are negative, and when $s<0$ all coefficients are positive. We will let $S_{2k}(r)$ indicate the $k$th partial sum--ie when we have $k+1$ terms terminating at $r^{2k}$.

\s We define \begin{equation}
F(r)=||\mbx+r\mby||_s^s+||\mbx-r\mby||_s^s=\sum_{i=1}^n |x_i+ry_i|^s+|x_i-ry_i|^s.
\end{equation}
We see that \begin{equation}
\frac{d^k}{dr^k}F(r)=(s)_{k}\sum_{i=1}^n y_i^k\left(\sgn(x_i+ry_i)^k|x_i+ry_i|^{s-k}+(-1)^k\sgn(x_i-ry_i)^k|x_i-ry_i|^{s-k}  \right)
\end{equation}
We claim that when $0<s<1$, $\frac{d^k}{dr^k}F(r)\leq 0$ for $0<r<1$, and when $s<0$, $\frac{d^k}{dr^k}F(r)\geq0$. To see this, we see that in both cases, when $k$ is odd, $\frac{d^k}{dr^k}F(r)|_{r=0}=0$. When $k$ is even, we have the simplified representation \begin{equation}
\frac{d^k}{dr^k}F(r)=
(s)_{k}\sum_{i=1}^n y_i^k\left(|x_i+ry_i|^{s-k}+|x_i-ry_i|^{s-k}  \right).
\end{equation}
Clearly, this will always match the sign of $(s)_{k}$: so it is negative when $0<s<1$, and positive when $0<s$. Therefore, $\frac{d^k}{dr^k}F(r)$ must also have the desired sign when $k$ is odd.

\s We will first consider the $0<s<1$ case. We claim that \begin{equation}\label{si}
F(r)\Big|_{r=1}\leq S_{2k}(r)\Big|_{r=1}	
\end{equation}
for all $k$. We will prove this by looking at the $2k$th derivative. 

\s Clearly, \begin{equation}
\frac{d^{2k}}{dr^{2k}}S_{2k}(r)=s_{(2n)}||\mby||_s^{2k}||\mbx||_{s}^{s-2k}.	
\end{equation}
We see that \begin{align}
\frac{d^{2k}}{dr^{2k}}F(r)\big|_{r=0}&=2(s)_{2k} \sum_{i=1}^n y_i^{2k}|x_i|^{s-2k}\leq 2(s)_n||\mby||_s^{2k}||\mbx||_{s}^{s-2k}
\end{align}
by application of Holder's reverse inequality, and noting that $(s)_{2k}$ is negative. As $\frac{d^{2k+1}}{dr^{2k+1}}F(r)\leq0$, this inequality must continue to hold for the full desired range of $r$. Taking the limit of the partial sums proves the theorem.

\s When $s<0$, the same proof can be repeated, with the inequalities throughout reversing because $(s)_{2k}$ is positive.
\end{proof}

\s Extending this matrices requires only the majorization of Theorem \ref{matdiag}: 

\begin{proof}{Proof of Theorem \ref{tHR}}
We will give the proof for $0<s<1$; the proof for $s<0$ proceeds similarly. Once more we assume $||C||_s>||D||_s$. The primary methodology in this section will be a full term-by-term comparison of Taylor representations of $\Tr[(C+rD)^s+(C-rD)^s]$ and $(||C||_s+r||D||_s)^s+(||C||_s+r||D||_s)^s$. Clearly, as $||C||_s>||D||_s$, the latter's Taylor series is convergent at $r=1$.  

\s For a positive matrix $X$ and $0<s<1$, for positive normalization constant $c_s$ we have \begin{equation}\label{pms}
	X^s=c_s\int_0^\infty\left(\frac{1}{t}-\frac{1}{t+X} \right)t^s dt
\end{equation}

Making the standard substitution $H=C+t$, $K=H^{-1/2}DH^{-1/2}$ then using the fact that $(I\pm K)^{-1}$ can be written as a power series, we see that \begin{align}
||C+rD||_s^s+||C-rD||_s^s&=c_s\int_0^\infty\Tr\left(\frac{1}{t}-\frac{1}{t+C+rD}-\frac{1}{t+C-rD}  \right)t^s dt\\
	&=c_s\int_0^\infty\Tr\left(\frac{1}{t}- H^{-1/2}\left(\sum_{k=0}^{\infty}K^{2k}r^{2k} \right)H^{-1/2} \right)t^p dt
\end{align}
Then as $C+D, C-D\geq 0$, it is clear that \begin{equation}
	H^{-1/2}\left(\sum_{k=0}^{\infty}K^{2k} \right)H^{-1/2}=	\frac{1}{t+C+D}+\frac{1}{t+C-D}
\end{equation}
is indeed a valid representation, and so the lefthand side can be expressed as a Taylor series that is also convergent at $r=1$. We can proceed to examine the two Taylor series term by term. 

\s We first write out the Taylor series representation \begin{align}
	(||C||_s+r||D||_s)^s+(||C||_s-r||D||_s)^s&=2\sum_{k=0}^{\infty} \frac{(s)_{2k}}{2k!} ||D||_s^{2k}	||C||_s^{s-2k} r^{2k}.\label{16s}
\end{align}
Note that like in the vector case, the first coefficient is always positive, and afterwards the coefficients are always negative.

\s For $||C+rD||_s^s+||C-rD||_s^s$, we use the trace to explicitly write out derivatives \begin{align}
	&\frac{d}{dr}\Tr[(C+rD)^s+(C-rD)^s]\Big|_{r=0}=s\Tr[(C+rD)^{s-1}D-(C-rD)^{s-1}D]\Big|_{r=0}=0 \\
	&\frac{d^k}{dr^k} \Tr[(C+rD)^s+(C-rD)^s]\Big|_{r=0}=-sc_sk!\int_0^\infty t^{s-1}\Tr\left[\left((1+(-1)^{k+1})\frac{1}{C+t}D \right)^k \right]dt
\end{align}
We note \begin{equation}
	\Tr\left[\left((1+(-1)^{k+1})\frac{1}{C+t}D \right)^k \right]=\begin{cases} 0 & k \text{ is odd} \\ 2\Tr\left[\left(\frac{1}{C+t}D \right)^k \right] & k \text{ is even} \end{cases}
\end{equation}
so we only need to concern ourselves with even $k$. 

\s We choose the basis such that $D$ is diagonal. We now claim that \begin{equation}
	\Tr\left[\left(\frac{1}{C+t}D \right)^{2k}  \right]\geq \Tr\left[\left(\frac{1}{C_{\Diag}+t}D \right)^{2k}  \right]
\end{equation}
Without loss of generality, we can assume that $D$ is invertible. We re-arrange \begin{equation}\label{r1s}
	\Tr\left[\left(\frac{1}{C+t}D \right)^{2k}  \right]=	\Tr\left[\left((C+t)^{-1/2}D(C+t)^{-1}D (C+t)^{-1/2} \right)^{k}  \right]
\end{equation}

\s This is in the form of the special function
\begin{equation}
	\Psi(A,B,K)_{q,r,s}:=\Tr\left[\left(B^{\frac{q}{2}}KA^pK^\ast B^{\frac{q}{2}}\right)^s  \right]	
\end{equation}
whose convexity and concavity properties have been studied in great detail (see \cite{Carlen2018}, \cite{Zhang2020} for a full treatment).  In particular, it was proven in \cite{Hiai2016}  for $A,B, \geq 0$ and any invertible $K$ that $\Psi$ is jointly convex in $A$ and $B$ when $-1\leq q\leq p\leq 0$ and $s>0$. We choose $A=B=(C+t)$, $q=p=-1$, and $K=D$. Then the right hand side of Equation \ref{r1s} is convex in $C+t$, and hence replacing $(C+t)$ with $(C+t)_{\Diag}=C_{\Diag}+t$ will only decrease the trace. 

\s By direct integration \begin{align}
	sc_s(2k)!\int_{0}^\infty t^{s-1}\Tr\left[\left(\frac{1}{C_{\Diag}+t}D \right)^{2k}  \right]dt&=(s)_{2k} \sum_{i=1}^n |c_{ii}|^{s-2k}\lambda_i(D)^{2k}	\\
	&\geq (s)_{2k} \left(\sum_{i=1}^{n}(c_{ii}^{s-2k})^{\frac{s}{s-2k}}  \right)^{\frac{s-2k}{s}}\left(\sum_{i=1}^{2n}(\lambda_i(D)^{2k})^{\frac{s}{2k}} \right)^{\frac{2k}{s}} \\
	&=(s)_{2k}||C_{\Diag}||_s^{s-2k}||D||_{s}^{2k} \label{f2s}\\
	&\geq (s)_{2k}||C||_s^{s-2k}||D||_s^{2k} \label{f3s}
\end{align}
When doubled and divided by the appropriate factorial, Line \ref{f3s} is precisely the $2k$th coefficient in the power series expansion in Equation (\ref{16s}). Recalling the negative sign, we compare term-by-term the Taylor series at $r=1$, and we see that the inequality \begin{equation}
	||C+D||_p^p+||C-D||_p^p\leq (||C||_p+||D||_p)^p+(||C||_p-||D||_p)^p	
\end{equation}
must hold. 

\s When $D\geq 0$, the proof simplifies to one entirely reliant on majorization. We note that for a positive matrix $X$, by majorization gives $||X_{\Diag}||_s^s\geq ||X||_s^s$ for $0<s<1$, and hence $||X_{\Diag}||_s\geq ||X||_s$. Furthermore, for  $0<s<1$, the function \begin{equation}\label{sfunc}
(x+y)^s+|x-y|^s
\end{equation}
is strictly decreasing in $x$ for fixed $y$ when $x\leq y$, and strictly increasing in $x$ when $x>y$. 

\s We consider $C+D, C-D$ in the basis where $C$ is diagonal. Then we note as $C+D_{\Diag}, C-D_{\Diag}\geq 0$, we have $||C||_s\geq ||D_{\Diag}||_s\geq ||D||_s$. Then treating $||D_\Diag||_s$ and the $x$ of Equation (\ref{sfunc}), and applying first majorization then Theorem \ref{RHFV} \begin{align}
||C+D||_s^s+||C-D||_s^s&\leq  ||C+D_{\Diag}||_s^s+||C-D_{\Diag}||_s^s\label{af}	\\
&\leq (||C||_s+||D_\Diag||_s)^s+(||C||_s-||D_\Diag||_s)^s \\
&\leq (||C||_s+||D||_s)^s+(||C||_s-||D||_s)^s \label{ag}
\end{align}
The fact that Equation (\ref{sfunc}) is strictly decreasing means that a requirement for equality is $D=D_\Diag$, or that $C$ and $D$ commute. 

\s When $s<0$, we now have $||X_{\Diag}||_s^s\leq ||X||_s^s$, and the structure of Equation (\ref{sfunc}) reverses: it is now strictly increasing in $x$ for fixed $y$ when $x\leq y$, and strictly decreasing in $x$ when $x>y$.  Then the arguments of Lines (\ref{af})-(\ref{ag}) are reversed, once more with equality requiring $D=D_\Diag$. 

\s We note in this simplified proof that we only leveraged the positivity condition $C+D, C-D\geq 0$ and $D\geq 0$ to guarantee 
the correct relationship between $||C\pm D||_s^s\leq ||C+\pm D_{\Diag}||_s$ and
$||C||_s\geq ||D_{\Diag}||_s\geq ||D||_s$ for $0<s<1$, reversing when $s<0$. For any $X,Y$ whose norms fulfill those inequalities, the above proof still holds. However, we can use the following incredibly useful property of $2\times 2$ self-adjoint matrices to construct counterexamples:

\begin{lemma}\label{2b2}
Let $A\in M_{2\times 2}(\C)$ be a self-adjoint matrix with eigenvalues $(\lambda_1(A), \lambda_2(A))$. Let $\mbb\in \R^2$. Then we can uniquely establish the matrix $B$ with eigenvalues $\mbb$ of \begin{equation}
B=\begin{pmatrix} x_1 & x_2 \\ \overline{x}_2 & x_3 \end{pmatrix}	
\end{equation}
up to $|x_2|$ with just the selection of one desired eigenvalue of the matrix $A+B$, provided that $\lambda_1(A)\neq \lambda_2(A)$, $\lambda_1(B)\neq \lambda_2(B)$, and the chosen eigenvalue $\lambda_i(A+B)$ satisfies $\lambda_2(A)+\lambda_2(B)\leq\lambda_i(A+B)\leq\lambda_1(A)+\lambda_1(B)$, and either $\lambda_i(A+B)\geq \lambda_1(A)+\lambda_2(B), \lambda_2(A)+\lambda_1(B)$ or $\lambda_i(A+B)\leq \lambda_1(A)+\lambda_2(B), \lambda_2(A)+\lambda_1(B)$
\end{lemma}
\begin{proof}
We first comment on where our conditions come from: if $A$ or $B$ is a multiple of the identity, $A$ and $B$ \textit{must} commute, which determines $\lambda_i(A+B)=\lambda_i(A)+\lambda_i(B)$, so we are not free to choose $\lambda_i(A+B)$. If $\lambda_i(A+B)$ does not satisfy $\lambda_2(A)+\lambda_2(B)\leq\lambda_i(A+B)\leq\lambda_1(A)+\lambda_1(B)$, then $\lambda(A+B)\prec\lambda(A)+\lambda(B)$ is violated, which we also know cannot be true, so no matrices $A$ and $B$ exist with those eigenvalue relationships.	Finally, we know from the re-arrangement of the same majoriazation identity that $\lambda(A+B)\succ\lambda_{\uparrow}(A)+\lambda_{\downarrow}(B)$. The third condition enforces this, with the two possibilities reflecting the choices $i=1,2$ and $\lambda_1(A)+\lambda_2(B)\geq \lambda_2(A)+\lambda_1(B)$ or vice-versa.

\s Without loss of generality, we assume we are in a basis such that $A=[\lambda(A)]$. There are three free coordinates for $B$ in this basis, and the fourth is determined by self-adjointness. To find them, we must solve the system of equations \begin{align}
&\det\left(\begin{bmatrix} x_1-\lambda_1(B) & x_2 \\ \overline{x}_2 & x_3-\lambda_1(B) \end{bmatrix}\right)=0, \\	&\det\left(\begin{bmatrix} x_1-\lambda_2(B) & x_2 \\ \overline{x}_2 & x_3-\lambda_2(B) \end{bmatrix}\right)=0 \\
&\det\left(\begin{bmatrix} x_1+\lambda_1(A)-\lambda_i(A+B) & x_2 \\ \overline{x}_2 & x_3+\lambda_2(A)-\lambda_i(A+B) \end{bmatrix}\right)=0
\end{align}
This can be directly solved, with the single solution that satisfies majorization conditions and hence represents an admissible matrix $B$ of \begin{align}
\begin{split}	
&x_1=\frac{\lambda_1(A)(\lambda_2(A)+\lambda_1(B)+\lambda_2(B))+\lambda_1(B)\lambda_2(B)}{\lambda_1(A)-\lambda_2(A)}\\
&\qquad \qquad +\frac{\lambda_i(A+B)(\lambda_i(A+B)-\lambda_1(A)-\lambda_2(A)-\lambda_1(B)-\lambda_2(B))}{\lambda_1(A)-\lambda_2(A)}
\end{split}\\
\begin{split}
&x_2=\text{RootOf}(z\overline{z}(\lambda_1(A)-\lambda_2(A))^2+ \\ 
&\qquad\qquad(\lambda_(A)+\lambda_1(B)-\lambda_i(A+B)) (\lambda_2(A)+\lambda_2(B)-\lambda_i(A+B))* \\
&\qquad \qquad(\lambda_1(A)+\lambda_2(B)-\lambda_i(A+B)) (\lambda_2(A)+\lambda_1(B)-\lambda_i(A+B))\end{split}\\
\begin{split}
x_3=-&\frac{\lambda_2(A)(\lambda_1(A)+\lambda_1(B)+\lambda_2(B))+\lambda_1(B)\lambda_2(B)}{\lambda_1(A)-\lambda_2(A)}\\
&\qquad  +\frac{\lambda_i(A+B)(\lambda_i(A+B)-\lambda_1(A)-\lambda_2(A)-\lambda_1(B)-\lambda_2(B))}{\lambda_1(A)-\lambda_2(A)}\\
\end{split}
\end{align}
Note that $x_2$ exists only when the quantity \begin{equation}
\begin{split}
&(\lambda_(A)+\lambda_1(B)-\lambda_i(A+B)) (\lambda_2(A)+\lambda_2(B)-\lambda_i(A+B))* \\
&\qquad\qquad\qquad\qquad\qquad (\lambda_1(A)+\lambda_2(B)-\lambda_i(A+B)) (\lambda_2(A)+\lambda_1(B)-\lambda_i(A+B))\end{split}
\end{equation}
is negative. The requirement $\lambda_2(A)+\lambda_2(B)\leq\lambda_i(A+B)\leq\lambda_1(A)+\lambda_1(B)$ ensures that the first term in the product is positive, and the second term is negative.  The requirement either $\lambda_i(A+B)\geq \lambda_1(A)+\lambda_2(B), \lambda_2(A)+\lambda_1(B)$ or the opposite enforces that the third and fourth terms are the same sign, so $x_2$ is defined. Finally, we see that this expression implies our choice of $x_2$ is unique up to $|x_2|$.
\end{proof}

\begin{figure}\label{cexf}
	\includegraphics[width=\textwidth]{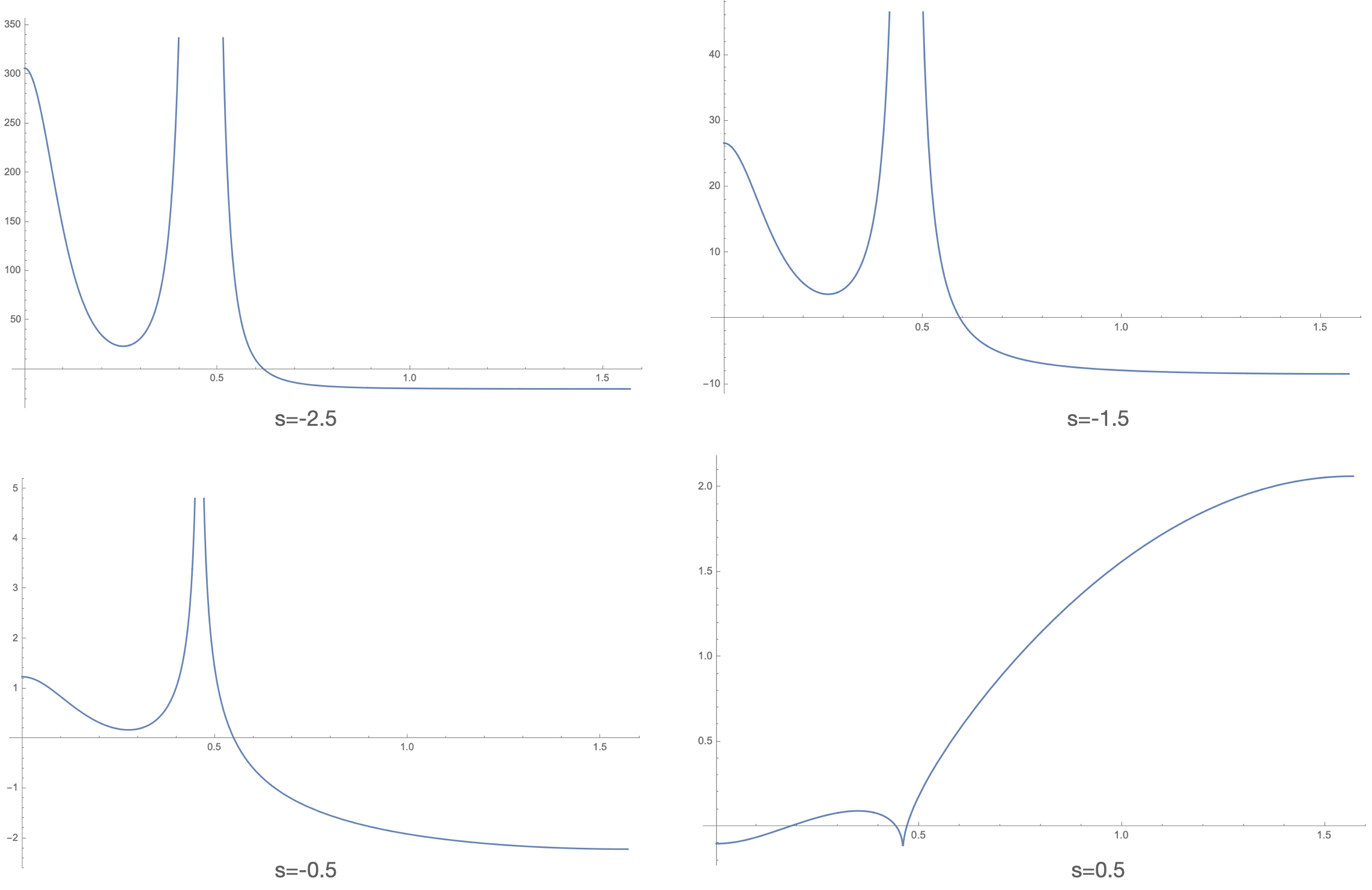}
	\caption{Comparison of Line (\ref{cex}) as a function of $0<t<\pi$ for self-adjoint matrices $A,B$ with $\lambda(A)=(-3,-5.5)$, $\lambda(B)=(3.4,-5.6)$.}    
\end{figure}

\s With Lemma \ref{2b2}, we can easily calculate the full range of possible eigenvalues $\lambda(A+B)$ and $\lambda(A-B)$ for $2\times 2$ self-adjoint matrices with pre-determined eigenvalues. It turns out that these eigenvalues will depend on $x_1$, $x_3$, and $|x_2|$, which now \textit{are} unique. Therefore, the full picture for Hanner-like inequalities for $2\times 2$ self-adjoint matrices can be determined by calculating the eigenvalues of  \begin{equation}
\begin{pmatrix} \lambda_1(A) & 0 \\ 0 & \lambda_2(A) \end{pmatrix} \pm \begin{pmatrix} \cos(t) & -\sin(t) 	\\ \sin(t) & \cos(t) \end{pmatrix} \begin{pmatrix} \lambda_1(B) & 0 \\ 0 & \lambda_2(B) \end{pmatrix}\begin{pmatrix} \cos(t) & \sin(t) 	\\ -\sin(t) & \cos(t) \end{pmatrix} 
\end{equation}
with $0\leq t\leq \pi$. Setting up \textit{Mathematica} code that allows for the manual manipulation of $\lambda(A)$ and $\lambda(B)$ then plotting \begin{equation}\label{cex}
||A+B_t||_s^s+||A-B_t||_s^s-(||\lambda(A)||_s+||\lambda(B)||_s)^s-\big|||\lambda(A)||_s-||\lambda(B)||_s\big|^s
\end{equation}
 reveals that, for example, the choice of $\lambda(A)=(-3,-5.5)$, $\lambda(B)=(3.4,-5.6)$ has the property that the sign of Line (\ref{cex}) depends on $t$. Figure (\ref{cexf}) shows this plot dependent on $t$ for various values of $s$.
\end{proof}

\section{Singular Value Rearrangement Inequalities}
\begin{proof}{Proof of Theorem \ref{svr1}}

\s Without loss of generality we assume that $C+D,C-D>0$; then a limiting argument extends to the positive semidefinite case. For a positive matrix $X$ with positive normalization constant $k_s$ we can use spectral calculus to write \begin{align}\label{sint}
X^s=\begin{cases}
c_s \int_0^\infty t^s \left(\frac{1}{t}-\frac{1}{t+X} \right)	dt & 0<s<1 \\
c_s \int_0^\infty t^s \left(\frac{1}{t+X} \right)	dt  	& -1<s <0 \\
c_s \int_0^\infty t^{\ubar{s}-s} \left(\frac{1}{t+X} \right)^{|\ubar{s}|}	dt  & s<-1
\end{cases}	
\end{align}
where $\ubar{s}$ and $\overline{s}$ will be used to indicate the floor and ceiling of $s$ respectively.

\s For $-1<s<0$ and $0<s<1$, we can make the standard substitution $H=C+t$, $K=H^{-1/2}DH^{-1/2}$ then using the fact that $(I\pm K)^{-1}$ can be written as a power series, to write \begin{align}
\frac{1}{t+C+D}+\frac{1}{t+C-D}&=H^{-1/2}\left(\sum_{k=0}^{\infty}(-1)^kK^{k} \right)H^{-1/2}+H^{-1/2}\left(\sum_{k=0}^{\infty}K^{k}\right)H^{-1/2} \\
&=H^{-1/2}\left(\sum_{k=0}^{\infty}K^{2k}\right)H^{-1/2}
\end{align}

\s To each term, we can apply the majorization identities Theorem \ref{sabweak} and Lemmas \ref{weakpower} and \ref{weaksum},  \begin{align}
\Tr[H^{-1/2}K^{2k}H^{-1/2}]&\leq \sum_{i=1}^n (\sigma_i(H^{-1})\sigma_i(K))^{2k} \\&\leq \sum_{i=1}^n \sigma_i(H^{-1})^{2k+1}\sigma_i(D)^{2k}\\&=\sum_{i=1}^n \sigma_{n+1-i}(C+t)^{2k+1}\sigma_i(D)^{2k}.
\end{align}

\s Combining each term back into the analogous expansion for $(t+\sigma_\uparrow(C)+\sigma_\downarrow(D))^{-1}+(t+\sigma_\uparrow(C)-\sigma_\downarrow(D))^{-1}$, we conclude that \begin{equation}
\Tr\left[\frac{1}{t+C+D}+\frac{1}{t+C-D} \right] \leq \Tr\left[\frac{1}{t+\sigma_\uparrow(C)+\sigma_\downarrow(D)}+\frac{1}{t+\sigma_\uparrow(C)-\sigma_\downarrow(D)} \right],
\end{equation}
and the theorem follows.

\s For the $s<-1$ case, we will return to the methodology of a term-by-term comparison of the derivatives of $F(r)=\Tr[(C+rD)^s+(C-rD)^s]$ and the Taylor representations of $||\sigma_\uparrow(C)+r\sigma_\downarrow(D)||^s+||\sigma_\uparrow(C)-r\sigma_\downarrow(D)||^s$ as introduced in the proof of Theorem \ref{RHFV}. Clearly, as $\sigma_n(C)\geq \sigma_1(D)$, the latter's Taylor series is convergent at $r=1$. It also has the property of only having nonzero even coefficients, which are non-negative. We calculate \begin{equation}
\frac{d^k}{dr^k}F(r)=(-1)^kc_s\int_0^\infty  t^{\ubar{s}-s} \Tr\left[\frac{1}{t+C+rD} \left(D\frac{1}{t+C+rD} \right)^{k+|\ubar{s}|}+(-1)^k\frac{1}{t+C-rD} \left(D\frac{1}{t+C-rD} \right)^{k+|\ubar{s}|}\right]	dt 
\end{equation}

\s We once more deduce the property that $\frac{d^k}{dr^k}F(r)\geq 0$ for all $k$, and that \begin{align}
\frac{d^k}{dr^k}F(r)\big|_{r=0}&=2c_s \int_0^\infty  t^{\ubar{s}-s} \Tr\left[\frac{1}{t+C} \left(D\frac{1}{t+C} \right)^{k+|\ubar{s}|}\right]	dt \\
&\leq 2c_s \int_0^\infty  t^{\ubar{s}-s} \sum_{i=1}^n\left(\sigma_i((C+t)^{-1})^{k+|\ubar{s}|+1}\sigma_i(D)^{k+|\ubar{s}|}\right)	dt\\
&= 2c_s \int_0^\infty  t^{\ubar{s}-s}\left(\sum_{i=1}^n \frac{\sigma_i(D)^{k+|\ubar{s}|}}{(\sigma_{n+1-i}(C)+t)^{k+|\ubar{s}|+1}} \right)	dt \label{rfinal}
\end{align}
where we can now see that Line (\ref{rfinal}) is the integral representation of the $2k$th Taylor coefficient of $||\sigma_\uparrow(C)+r\sigma_\downarrow(D)||^s+||\sigma_\uparrow(C)-r\sigma_\downarrow(D)||^s$. The same partial sums argument yields the desired inequality. 

\s When the conditions on $C$ and $D$ are not met, the method used to find counterexamples when the conditions of Thereom (\ref{tHR}) do not hold can also be used to produce counterexamples to the rearrangement inequality.
\end{proof}

\begin{proof}{Proof of Theorem \ref{svr2}}
	
\s It is proven in \cite{Carlen2006} that \begin{equation}
	\Tr\left[\frac{1}{t+C+D}+\frac{1}{t+C-D}\right]\geq \Tr\left[\frac{1}{t+\sigma(C)+\sigma(D)}+\frac{1}{t+\sigma(C)-\sigma(D)}\right]
\end{equation}
for all $t$. Therefore, a direct application of the integral representations for $0<s<1$ and $-1<s<0$ yield the desired inequality. For $s<-1$, we apply the same Taylor expansion method as in Theorem \ref{svr1}, now noting that \begin{align}
\Tr\left[\frac{1}{t+C} \left(D\frac{1}{t+C} \right)^{k'}\right]	 
& =\Tr\left[\frac{1}{t+C}^{k'/2k'} \left(\frac{1}{t+C}^{1/2}D\frac{1}{t+C}^{1/2} \right)^{k'}\frac{1}{t+C}^{k'/2k'}\right]	\\
& \geq\Tr\left[\left(\frac{1}{t+C}^{(k'+1)/2k'}D\frac{1}{t+C}^{(k'+1)/2k'} \right)^{k'}\right]\\
&=\sum_{i=1}^n\sigma\left(\frac{1}{t+C}^{(k'+1)/2k'}D^{1/2}  \right)^{2k'} \\
&\geq\sum_{i=1}^n\sigma_i\left(\frac{1}{t+C}^{(k'+1)/2k'}\right)^{2k'}\sigma_{n+1-i}\left(D^{1/2}  \right)^{2k'} \\
&=\sum_{i=1}^n\frac{\sigma_i(D)^{k'}}{(\sigma_i(C)+t)^{k'+1}}.
\end{align}

\s To see that the conditions on $C$ and $D$ are necessary, take $D$ unitary and $C>I$. Then $\sigma_\uparrow(D)=\sigma_\downarrow(D)$, and as Theorem \ref{svr2} holds. Examples can easily be found such that it holds strictly, so the inequality of Theorem \ref{svr1} is violated.
\end{proof}

\section{The Hanner Equality Cases For Matrices}

\begin{proof}{Proof of Theorem \ref{HE}}
We once again use a full term-by-term comparison of the Taylor series of  $\Tr[(C+rD)^p+(C-rD)^p]$ and $(||C||_p+r||D||_p)^p+(||C||_p+r||D||_p)^p$. Now we use the integral representation for a positive matrix $X$ and $1<p<2$, with positive normalization constant $c_p$, \begin{equation}\label{pm}
X^p=c_p\int_0^\infty\left(\frac{X}{t^2}+\frac{1}{t+X}-\frac{1}{t} \right)t^p dt.
\end{equation}

\s Once more, the Taylor series converges and the coefficients of $\Tr[(C+rD)^p+(C-rD)^p]$ can be directly calculated to only have nonzero even terms, and for said terms to here be greater than or equal to $ (p)_{2k}||C_{\Diag}||_p^{p-2k}||D||_p^{2k}$ which in turn is greater than or equal to  $(p)_{2k}||C||_p^{p-2k}||D||_p^{2k}$. 

\s To see that there is equality if and only if $|D|=kC$, the final noted inequality leverages that $||C||_p\geq ||C_{\Diag}||_p$ implies $||C||_p^{p-2k}\leq ||C_{\Diag}||_p^{p-2k}$. In fact, by as $|x|^p$ is strictly convex for $1<p<2$ we have by Theorem \ref{strictmaj} $||C||_p> ||C_{\Diag}||_p$ when $C\neq C_{\Diag}$. This immediately implies that there can only be equality when $C$ and $D$ commute. 

\s In the commuting case, \begin{align}	||C+D||_p^p+||C-D||_p^p&=||\lambda_i(C)+\lambda_{k_i}(D)||_p^p+||\lambda_i(C)-\lambda_{k_i}(D)||_p^p\\&=||\lambda_i(C)+|\lambda_{k_i}(D)|||_p^p+||\lambda_i(C)-|\lambda_{k_i}(D)|||_p^p
\end{align}
It was proven by Hanner originally in \cite{hanner1956}  that there is equality in the application of Hanner's inequality to sequences as above if and only if they are multiples of one another. Returning to the matrix expression, this gives the requirement on $D$ that $c|D|=C$. 

\s To establish the equality case for $C+D, C-D\geq 0$ when $p>2$, we must use the following dual representation: letting $q$ be the dual index of $p$, we see that \begin{equation}
(||C+D||_p^p+||C-D||_p^p)^{1/p}=\left|\left|\begin{pmatrix} C & D \\ D & C\end{pmatrix}\right|\right|_p=\Tr\left[\begin{pmatrix} C & D \\ D & C\end{pmatrix}\begin{pmatrix} X & Y \\ Y & X\end{pmatrix}  \right],
\end{equation}
where \begin{equation}
\left|\left|\begin{pmatrix} X & Y \\ Y & X\end{pmatrix}\right|\right|_q=1,
\end{equation}
and in fact is also positive semidefinite and commutes with $\begin{pmatrix} C & D \\ D & C\end{pmatrix}$. We therefore calculate \begin{align}
(||C+D||_p^p+||C-D||_p^p)^{1/p}&=2\Tr[CX+DY] \label{f}\\
&\leq 2(||C||_p||X||_q+||D||_p||Y||_q) \\
&=\Tr\left[\begin{pmatrix} ||C||_p & ||D||_p \\ ||D||_p & ||C||_p\end{pmatrix}\begin{pmatrix} ||X||_q & ||Y||_q \\ ||Y||_q & ||X||_q\end{pmatrix}  \right] \\
&\leq \left|\left|\begin{pmatrix} ||C||_p & ||D||_p \\ ||D||_p & ||C||_p\end{pmatrix}  \right|\right|_p  \left|\left| \begin{pmatrix} ||X||_q & ||Y||_q \\ ||Y||_q & ||X||_q\end{pmatrix} \right|\right|_q\label{sl} \\
&\leq \left|\left|\begin{pmatrix} ||C||_p & ||D||_p \\ ||D||_p & ||C||_p\end{pmatrix}  \right|\right|_p  \label{l}
\end{align}
where the inequality from Line (\ref{sl}) to Line (\ref{l}) leverages that $1\leq q\leq 2$, and hence \begin{equation}
1=||X+Y||_q^q+||X-Y||_q^q\geq (||X||_q+||Y||_q)^q+(||X||_q-||Y||_q)^q.
\end{equation}

\s For there to be equality in Hanner's Inequality, there must be equality in each of Lines (\ref{f}-\ref{l}). In particular, this means as we have just proven that $Y=c|X|$. We can change to a block diagonal basis to represent the original commutation was \begin{equation}
	\left[\begin{pmatrix} C + D & 0 \\0 &  C-D \end{pmatrix},  \begin{pmatrix} X + Y & 0 \\ 0&  X-Y\end{pmatrix}  \right]=0.
\end{equation}
or \begin{equation}
	[X+Y,C+D]=0, \qquad [X-Y,C-D]=0.	
\end{equation}	
When $X$ and $Y$ commute, then $[X+Y, X-Y]=0$, and therefore $[C+D,C-D]=0$. We write this out explicitly to see \begin{align}
(C+D)&(C-D)=(C-D)(C+D) \\
&\Rightarrow C^2+DC-CD-D^2=C^2-DC+CD-D^2 \\
& \Rightarrow [C,D]=-[D,C]=0. 	
\end{align}	
Then repeating the argument of the commuting case with Hanner's inequality for sequences before, there is once more equality if and only if  $c|D|=C$. 

\s For arbitrary $C,D$, we may first without loss of generality assume that $C,D$ are self-adjoint; otherwise we would apply the standard doubling technique noting the equality holds if and only if it holds for \begin{equation}
\widehat{C}=\begin{bmatrix} 0 & C \\ C^\ast & 0 \end{bmatrix}, \qquad 	\widehat{D}=\begin{bmatrix} 0 & D \\ C^\ast & 0 \end{bmatrix}.
\end{equation}	
Hanner's inequality is deduced in the $p\geq 4$ range using the fact that \begin{equation}
\left|\left|\begin{pmatrix} C & D \\ D & C\end{pmatrix}\right|\right|_p^2=
	\left|\left|\begin{pmatrix} C^2+D^2 & CD+DC \\ CD+DC & C^2+D^2\end{pmatrix}\right|\right|_{p/2},
\end{equation}
and taking the dual representation with a positive semidefinite and commuting block matrix.  As a Hermitian matrix and its square are mutually diagonalizable, we note that \begin{equation}
	\left[\begin{pmatrix} C & D \\ D & C\end{pmatrix}^2,  \begin{pmatrix} X & Y \\ Y & X\end{pmatrix}  \right]=0	\qquad \text{ iff } \qquad \left[\begin{pmatrix} C & D \\ D & C\end{pmatrix},  \begin{pmatrix} X & Y \\ Y & X\end{pmatrix}  \right]=0,
\end{equation}	
and therefore we can draw the same conclusion $[C,D]=0$ from $[X,Y]=0$. Again by Hanner's inequality on $\ell^p$, then $c|D|=|C|$. For the $1\leq p\leq \frac{4}{3}$, there can equality in Hanner's inequality for operators $u,v$ once more only if there is equality for the dual operators $\phi$ associated with $(u+v)$ and $\psi$ associated with $(u-v)$. These dual operators once more can be represented in this case by a commuting block matrix, and the commutation argument is repeated. 
\end{proof}

\s We note that although we do not formulate the equality case for $s<1$, that the same Taylor expansion argument holds using integral representations for $s<1$, so this theorem can also be applied to $X+Y, X-Y\geq 0$, $s<1, s\neq 0$.

\section*{Acknowledgements}
This research was funded by the NDSEG Fellowship, Class of 2017.  Thank you to my advisor, Professor Eric Carlen, for bringing my attention to the problem and providing me with a background to the subject. Thank you to Professor Jean-Christophe Bourin for discussion on the variational representations.

\bibliographystyle{spmpsci}   
\bibliography{references} 
\end{document}